\documentclass{amsart}
\usepackage[latin1]{inputenc}
\usepackage{subfigure}
\usepackage[round]{natbib}
\usepackage{amsmath,amssymb,amscd, amstext, amsopn, amsxtra}
\usepackage[svgnames]{xcolor}
\usepackage{tikz}
\usepackage{pgf}
\usepackage{graphicx}
\pgfdeclarelayer{background layer}
\pgfdeclarelayer{foreground layer}
\pgfsetlayers{background layer,main,foreground layer}

\usepackage{graphics}
\usepackage{epsfig}

\usepackage[round]{natbib}

\usetikzlibrary{decorations}
\usepgflibrary{decorations.pathmorphing} 
\usetikzlibrary{decorations.pathmorphing} 
\usetikzlibrary{arrows}

\numberwithin{equation}{section}
\newtheorem{theorem}{Theorem}[section]

\newtheorem{proposition}[theorem]{Proposition}
\newtheorem{lemma}[theorem]{Lemma}

\newtheorem {corollary}[theorem]{Corollary}
\newtheorem{definition}[theorem]{Definition}
\newtheorem{example}[theorem]{Example}
\newtheorem{remark}[theorem]{Remark}

\newcommand{\ZZ}{{\mathbb Z}} 

\newcommand{\e}{{\varepsilon}}  

\def\reals{\mathbb {R}}
\def\ints{\mathbb {Z}}

\def\m{{m}}
\def\n{{n}}

\newcommand{\brac}[2]{{\langle #1\mid#2\rangle}}

\def\length{\ell}

\newcommand{\core}[1]{{{\mathcal C}_{#1}}}

\newcommand{\shi}[3]{{{\mathcal S}_{#2,#3}^{#1}}}

\def\shinm{\shi{}{n}{m}}
\newcommand{\region}[1]{{R}_{#1}}

\def\tR{\tilde{\region{}}}

\def\R1{\region{1}}

\newcommand{\A}{\mathcal{A}}
\newcommand{\Afund}{\A_0}
\newcommand{\domAlcoves}{\mathfrak{A}}
\newcommand{\wA}{w^{-1}\Afund}
\newcommand{\kwa}{k_{w, \a}}
\newcommand{\kwaij}[2]{k_{w, \a_{#1,#2}}}

\newcommand{\FVMap}{\Phi}
\newcommand{\RiMap}{\Psi}
\def\tp{\tilde{p}}
\newcommand{\FVMapArg}[1]{\FVMap(#1)}
\newcommand{\RiMapArg}[1]{\RiMap(#1)}

\newcommand{\mapFour}[3]{\mu^#1_{#2,#3}} 

\def\mapFourBn1{\mapFour{B}{n}{1}}

\newcommand{\roots}{\Delta}
\newcommand{\posRoots}{\roots^{+}}
\newcommand{\negRoots}{\roots^{-}}
\newcommand{\simpleRoots}{\Pi}
\def\a{\alpha}
\def\aij{\alpha_{ij}}
\def\ai{\alpha_{i}}

\def\Sn{\mathfrak{S}_n}
\def\affS{{\widehat{\mathfrak{S}}}_n}
\newcommand{\affSn}[1]{{\widehat{\mathfrak{S}}}_{#1}}
\newcommand{\coset}{{\affSn{n} / \Sn }}
\newcommand{\alij}[2]{\a_{#1, #2}}
\newcommand{\ei}{{\e_i}}  

\newcommand{\Hak}[2]{{H}_{#1, #2}}
\newcommand{\HH}[3]{\Hak{\alpha_{#1#2}}{#3}}

\newcommand{\Hp}[2]{{{\Hak{#1}{#2}}^+}}

\newcommand{\Hn}[2]{{{\Hak{#1}{#2}}^-}}
\newcommand{\Halpha}[3]{{H}_{\alpha_{#1#2},#3}}
\newcommand{\Hyp}[1]{{\mathcal H}_{#1}}
\newcommand{\Hthm}{\Hak{\theta}{m}}

\newcommand{\sepSet}[3]{{\mathfrak h}^{#1}_{{#3}{#2}}}

\newcommand{\wP}[5]{F^{{#1}}_{{#3}{#2}}(#4,#5)}

\newcommand{\wPthm}{F^{n}_{\theta,m}(p,q)}
\newcommand{\rStat}[1]{\text{r}(#1)}
\newcommand{\cStat}[1]{\text{c}(#1)}
\newcommand{\polyInt}[2]{[#1]_{#2}}

\newcommand{\trunc}[3]{\left({#1}\right)_{\leq {#2}^{#3}}}
\newcommand{\colRem}[2]{\phi_{#1,#2}}

\newcommand{\ceil}[2]{\left\lceil {\frac{#1}{#2} } \right\rceil}

\newcommand{\floor}[1]{\left\lfloor {#1}  \right\rfloor}

\definecolor{Green}{rgb}{0.1,0,0.55}
\newcommand{\omitt}[1]{}
\newcommand{\fpsacOmitt}[1]{#1}
\newcommand{\fpsacInclude}[1]{}

\def\mvB{{\color{Red}$\blacktriangleright$}}
\def\mvEn{{\color{Red}$\blacktriangleleft$}}
\def\mvNew#1{\mvB #1\mvEn}%
\newcommand{\sfOmitt}[1]{}

\DeclareMathOperator{\Inv}{Inv}

\title[Counting Shi regions]{Counting Shi regions with a fixed separating
  wall}

\newcounter{r}
\newcounter{c}

\newcommand{\youngDiagram}[2]{
        \setcounter{r}{0}
        \foreach \row in #1 {
                \setcounter{c}{0}
                \foreach \b in \row {
                        \node at (\value{c}*#2 + .5*#2, \value{r}*#2+.5*#2) {\b};
                        \addtocounter{c}{1}

                }
                \draw[step = #2] (0,\value{r}*#2) grid (\value{c}*#2,\value{r}*#2+#2);
                \addtocounter{r}{-1}
        }
}

\author{Susanna Fishel}\thanks{Susanna Fishel was partially supported by Simons Foundation Grant 209806.} \address{School of Mathematical and
  Statistical Sciences, Arizona State University, Tempe, AZ 85287,
  USA}

\author{Eleni Tzanaki} \address{Department of Applied Mathematics,
  University of Crete, 71409 Heraklion, Crete, Greece}
\author{Monica Vazirani}\thanks{Monica Vazirani was partially supported by NSA grant H982300910076.}\address{Department of Mathematics, UC Davis,
  CA 95616, USA}

\begin{document}
\maketitle
\date{\today}

\begin{abstract}
Athanasiadis introduced separating walls for a region in the extended
Shi arrangement and used them to generalize the Narayana numbers. In
this paper, we fix a hyperplane in the extended Shi arrangement for
type $A$ and calculate the number of dominant regions which have the
fixed hyperplane as a separating wall; that is, regions where the
hyperplane supports a facet of the region and separates the region
from the origin.
\end{abstract}

\section{Introduction}
\label{sec:intro}
A hyperplane arrangement dissects its ambient vector space into
regions. The regions have walls--hyperplanes which support facets of
the region-- and the walls may or may not separate the region from the
origin. The regions in the extended Shi arrangement are enumerated by
well-known sequences: all regions by the extended parking function
numbers, the dominant regions by the extended Catalan numbers,
dominant regions with a given number of certain separating walls by
the Narayana numbers. In this paper we study the extended Shi
arrangement by fixing a hyperplane in it and calculating the number of
regions for which that hyperplane is a separating wall. For example,
suppose we are considering the $m$th extended Shi arrangement in
dimension $n-1$, with highest root $\theta$. Let $\Hthm$ be the $m$th
translate of the hyperplane through the origin with $\theta$ as
normal. Then we show there are $m^{n-2}$ regions which abut $\Hthm$
and are separated from the origin by it.

At the heart of this paper is a well-known bijection from certain
integer partitions to dominant alcoves (and regions). One particularly
nice aspect of our work is that we are able to use the bijection to
enumerate regions. We characterize the partitions associated to the
regions in question by certain interesting features and easily count
those partitions, whereas it is not clear how to count the regions
directly.

We give two very different descriptions of this bijection, one
combinatorial and one geometric.  The first description of the
bijection comes from group theory, from studying the Scopes
equivalence on blocks of the symmetric group. The second description
is the standard one for combinatorics and used when studying the
affine symmetric group. As a reference, it is good to have both forms
of this oft-used map in one place.  We can then prove several results
in two ways, using the different descriptions. Although
Theorem~\ref{thm:baseCase} is essentially the same as
Theorem~\ref{thm:altBaseCase}, the proofs are very different and
Propositions \ref{prop-gamma} and {\ref{prop-Gamma} are of independent
  interest.

We rely on work from several sources.  \cite{Shi1986} introduced what
is now called the Shi arrangement while studying the affine Weyl group
of type $A$, and \cite{S1998} extended it. We also use his study of
alcoves in \cite{Shi}. \cite{Richards}, on decomposition numbers for
Hecke algebras, has been very useful. The Catalan numbers have been
extended and generalized; see \cite{A2005a} for the
history. Fuss-Catalan numbers is another name for the extended Catalan
numbers. The Catalan numbers can be written as a sum of Narayana
numbers. \cite{A2005a} generalized the Narayana numbers. He showed
they enumerated several types of objects; one of them was the number
of dominant Shi regions with a fixed number of separating walls. This
led us to investigate separating walls. All of our work is for type
$A$, although Shi arrangements, Catalan numbers, and Narayana numbers
exist for other types.

In Section~\ref{sec:prelim}, we introduce notation, define the Shi
arrangement, certain partitions, and the bijection between them which
we use to count regions. In Section~\ref{sec:firstSepWall}, we
characterize the partitions assigned to the regions which have $\Hthm$
as separating wall. In order to enumerate the regions which have other
separating walls, we must use a generating function, which we
introduce in Section~\ref{sec:genfunc}. The generating function recods
the number of hyperplanes of a certain type in the arrangement which
separte the region from the origin. It turns out that in the base
case, where the hyperplane is $\Hthm$ these numbers can be easily read
from the region's associated $n$-core and we obtain
Corollary~\ref{cor:boundaryRegionsGF}:
$$\sum_{\substack{\region{}:\Hthm\textrm{ is a}\\\textrm{separating
      wall for }\region{}}}p^{\cStat{\region{}}}q^{\rStat{\region{}}} =
  p^mq^m[p^{m-1}+p^{m-2}q+\cdots+q^{m-1}]^{n-2}, $$ where $\cStat{}$
  and $\rStat{}$ are defined in Section~\ref{sec:genfunc}.  Finally,
  in Section~\ref{sec:sephyp}, we give a recursion for the generating
  functions from Section~\ref{sec:genfunc}, which enables us to count
  the regions which have other separating walls $\Hak{\a}{m}$.

\section{Preliminaries}
\label{sec:prelim}
Here we introduce notation and review some constructions.

\subsection{Root system notation}
\label{subsec:roots}
Let $\{ \e_1, \ldots, \e_\n \}$ be the standard basis of $\reals^n$
and $\brac{\,}{\,}$ be the bilinear form for which this is an
orthonormal basis.  Let $\a_i = \e_i - \e_{i+1}$. Then $\simpleRoots =
\{\a_1, \ldots, \a_{\n-1} \}$ is a basis of \[V = \{ (a_1, \ldots,
a_\n ) \in \reals^n \mid \sum_{i=1}^n a_i =0 \}.\] For $i\leq j$, we
write $\aij$ for $\alpha_i+\ldots+\alpha_j$. In this notation, we have
that $\alpha_{i}=\alpha_{ii}$, the highest root $\theta$ is
$\a_{1,n-1}$, and $\aij=\e_i-\e_{j+1}$.

The elements of $\roots = \{\e_i - \e_j \mid i \neq j\}$ are called
roots and we write a root $\alpha$ is positive, written $\alpha > 0$, if
$\alpha \in \posRoots = \{\e_i - \e_j \mid i < j\}$. We let $\negRoots
= - \posRoots$ and write $\alpha < 0$ if $\alpha \in \negRoots$. Then
$\simpleRoots$ is the set of simple roots.  As usual, we let $Q =
\bigoplus_{i=1}^{n-1} \ZZ \alpha_i$ be identified with the root
lattice of type $A_{\n-1}$ and $Q^+= \bigoplus_{i=1}^{n-1} \ZZ_{\ge 0}
\alpha_i$.

\subsection{Extended Shi arrangements}
\label{subsec:shi}

A {\it hyperplane arrangement} is a set of hyperplanes, possibly
affine, in $V$.  We are interested in certain sets of
hyperplanes of the following form. For each $\a\in\posRoots$, we
define the reflecting hyperplane
$$\Hak{\alpha}{0}  = \{ v \in V \mid \brac{v }{\alpha} = 0 \}$$
and  its $k$th translate, for $k\in\ints$,
$$\Hak{\alpha}{k}  = \{ v \in V \mid \brac{v }{\alpha} = k \}.$$

Note $\Hak{-\alpha}{-k} =\Hak{\alpha}{k} $ so we usually take $k \in
\ints_{\ge 0}$. Then the extended Shi arrangement, here called the
{\it $\m$-Shi arrangement}, is the collection of
hyperplanes \[\Hyp{\m} = \{ \Hak{\alpha}{k} \mid \alpha \in \posRoots,
-m < k \le m \}.\]

This arrangement is defined for crystallographic root systems of all
finite types.

{\it Regions\/} of the $\m$-Shi arrangement are the connected
 components of the hyperplane arrangement complement $V \setminus
 \bigcup_{H \in \Hyp{\m}} H$.

We denote the closed half-spaces $\{ v \in V \mid \brac{v}{\alpha} \ge
k \}$ and $\{ v \in V \mid\brac{v}{\alpha} \le k \}$ by
$\Hp{\alpha}{k}$ and $\Hn{\alpha}{k}$ respectively. The {\it dominant}
or {\it fundamental} chamber of $V$ is $\bigcap_{i=1}^{n-1}
\Hp{\alpha_i}{0}$. This paper primarily concerns regions and alcoves
in the dominant chamber.

A {\it dominant region} of the $\m$-Shi arrangement is a region that
is contained in the dominant chamber. We denote the collection of
dominant regions in the $\m$-Shi arrangement $\shinm$.

Each connected component of $$V \setminus \bigcup_{\stackrel{\alpha
    \in \posRoots}{k \in \ZZ}} \Hak{\alpha}{k}$$ is called an {\it
  alcove} and the {\it fundamental alcove} is $\Afund$, the interior
of $ \Hn{\theta}{1} \cap \bigcap_{i=1}^{n-1} \Hp{\alpha_i}{0}$.  A {\it
  dominant alcove} is one contained in the dominant chamber. Denote
the set of dominant alcoves by $\domAlcoves_n$.

A {\it wall} of a region is a hyperplane in $\Hyp{\m}$ which supports
a facet of that region or alcove. Two open regions are {\it separated}
by a hyperplane $H$ if they lie in different closed half-spaces
relative to $H$.  Please see \cite{A2005a} or \cite{hum90} for
details.  We study dominant regions with a fixed separating wall. A
{\it separating wall} for a region $R$ is a wall of $R$ which
separates $R$ from $\Afund$.

\subsection{The affine symmetric group}
\label{sec:affAlcoves}
\begin{definition}
The affine symmetric group, denoted $\affS$, is defined as
\begin{align*}
  \affS = \langle s_1, \ldots, s_{\n-1}, s_0 \mid
  s_i^2 = 1, \quad  & s_i s_j = s_j s_i \text{ if } i \not\equiv j \pm 1 \mod n, \\
  & 
  s_i s_j s_i = s_j s_i s_j \text{ if } i \equiv j \pm 1 \mod n
  \rangle
\end{align*}
for $n > 2$, but ${\affSn{2}} = \langle s_1, s_0  \mid s_i^2 = 1 \rangle$.
\end{definition}

The affine symmetric group contains the symmetric group $\Sn$ as a
subgroup.  $\Sn$ is the subgroup generated by the $s_i$, $0<i<n$.  We
identify $\Sn$ with the set of permutations of $\{1, \ldots, \n\}$ by identifying
$s_i$ with the simple transposition $(i,i+1)$.

The affine symmetric group $\affS$ acts freely and transitively on the
set of alcoves.  We thus identify each alcove $\A$ with the unique $w
\in \affS$ such that $\A = w^{-1} \Afund$.  Each simple generator $s_i$,
$i>0$, acts by reflection with respect to the simple root
$\alpha_i$. In other words, it acts by reflection over the hyperplane
$\Hak{\alpha_i}{0}$.  The element $s_0$ acts as reflection with
respect to the affine hyperplane $\Hak{\theta}{1}$.

More specifically, the action on $V$ is given by
\begin{gather*}
s_i (a_1, \ldots, a_i, a_{i+1}, \ldots, a_\n) = (a_1, \ldots, a_{i+1},
a_i,\ldots, a_\n) \quad \text{ for $i \neq 0$, and} \\ s_0 (a_1,
\ldots, a_\n) = (a_\n +1, a_2 , \ldots, a_{\n-1}, a_1 - 1).
\end{gather*}
Note $\Sn$ preserves $\brac{\;}{\;}$, but $\affS$ does not.

\subsection{Shi coordinates and Shi tableaux.}
\label{sec;shiCoords}
\label{sec;shiTabs}
Every alcove $\A$ can be written as $w^{-1}\Afund$ for a unique
$w\in\affS$ and additionally, for each $\a\in\posRoots$, there is a
unique integer $k_{\a}$ such that $k_{\a}<\brac{\a}{x}<k_{\a}+1$ for
all $x\in\A$. Shi characterized the integers $k_{\a}$ which can arise
in this way and the next lemma gives the conditions for type $A$.
\begin{lemma}[\cite{Shi}] \label{lem:alcoveCoords}Let
  $\{k_{\aij}\}_{1\leq i\leq j\leq n-1}$ be a set of ${n\choose 2}$
  integers. There exists a $w\in\affS$ such that
$$k_{\aij}<\brac{\aij}{x}<k_{\aij}+1$$ for all $x\in w^{-1}\Afund$ if and
only if $$k_{\a_{it}}+k_{\a_{t+1,j}}\leq k_{\aij}\leq k_{\a_{it}}+k_{\a_{t+1,j}}+1,$$ for
all $t$ such that $i\leq t<j$.
\end{lemma}

From now on, except in the discussion of
Proposition~\ref{prop:cAndr2}, we write $k_{ij}$ for $k_{\aij}$.
These $\{k_{ij}\}_{1\leq i\leq n-1}$ are the {\em Shi coordinates} of
the alcove. We arrange the coordinates for an alcove $\A$ in the
Young's diagram (see Section~\ref{sec:partitions}) of a staircase
partition $(n-1,n-2,\ldots,1)$ by putting $k_{ij}$ in the box in row
$i$, column $n-j$. See \cite{KOP} for a similar arrangement of sets
indexed by positive roots. For a dominant alcove, the entries
are nonnegative and non-increasing along rows and columns.

We can also assign coordinates to regions in the Shi arrangement. In
each region of the $\m$-Shi hyperplane arrangement, there is exactly
one ``representative,'' or $\m$-minimal, alcove closest to the
fundamental alcove $\Afund$. See \cite{Shi1987} for $m=1$ and
\cite{A2005a} for $m\geq 1$. Let $\A$ be an alcove with Shi
coordinates $\{k_{ij}\}_{1\leq i\leq n-1}$ and suppose it is the
$\m$-minimal alcove for the region $\region{}$. We define coordinates
$\{e_{ij}\}_{1\leq i\leq j\leq n-1}$ for $\region{}$ by
$e_{ij}=\min(k_{ij},\m)$.

Again, we arrange the coordinates for a region $\region{}$ in the Young's
diagram (see Section~\ref{sec:partitions}) of a staircase partition
$(n-1,n-2,\ldots,1)$ by putting $e_{ij}$ in the box in row $i$,
column $n-j$. For dominant regions, the entries are nonnegative and non-increasing
along rows and columns.

\begin{example}
For $n=5$, the coordinates are arranged\\
\begin{tikzpicture}[font=\footnotesize]
\begin{scope}[shift={(-5,0)}]
\youngDiagram{{{$k_{14}$, $k_{13}$, $k_{12}$, $k_{11}$},{$k_{24}$, $k_{23}$,$k_{22}$},{$k_{34}$,$k_{33}$},{$k_{44}$}}}{.5}
\end{scope}
\begin{scope}[shift={(0,0)}]
\youngDiagram{{{$e_{14}$, $e_{13}$, $e_{12}$,
    $e_{11}$}, {$e_{24}$, $e_{23}$,
    $e_{22}$},{$e_{34}$, $e_{33}$}, {$e_{44}$}}}{.5}
\end{scope}
\end{tikzpicture}
\end{example}

\begin{example}
The dominant chamber for the $2$-Shi arrangement for $n=3$ is illustrated
in Figure~\ref{fig:alcoveRegionExample} The yellow region has
coordinates $e_{12}=2$, $e_{11}=1$, and $e_{22}=2$. Its $2$-minimal
alcove has coordinates $k_{12}=3$, $k_{11}=1$, and $k_{22}=2$.
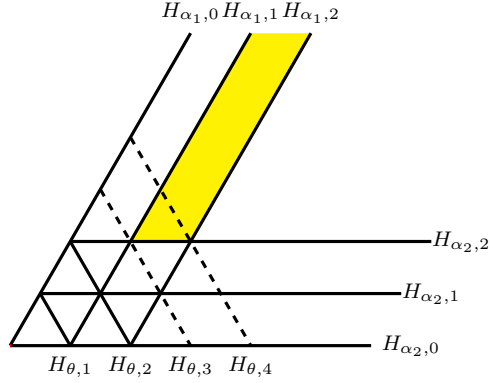
\begin{figure}[ht]
\begin{center}
\begin{tikzpicture}[scale=.8]

\begin{scope}

\tikzstyle{every node}=[font=\footnotesize]
[4,1,2]

\def\a{1.73205}

\filldraw[fill = yellow, fill opacity = 1, draw opacity = 0] 
(2,\a)--(4,3*\a)--(5,3*\a)--(3,\a)--cycle;

\path [draw = black, very thick, draw opacity = 1, shift = {(0,0)}]
(0,0) -- +(60:6) node[black,above] {$H_{\alpha_1,0}$};

\path [draw = black,very thick, draw opacity = 1] (0,0) -- +(6,0) node[black,right] {$H_{\alpha_2,0}$};

\path [draw = red,very thick,draw opacity = 1] (0,0) -- (0,0);

\path [draw = black, very thick, draw opacity = 1, shift = {(1,0)}]
(0,0) -- +(60:6) node[black,above] {$H_{\alpha_1,1}$};

\path [draw = black,very thick, draw opacity = 1] (0.5,0.866025) -- +(6,0);
\draw [shift = {(0.5,0.866025)}] (6.5,0) node {$H_{\alpha_2,1}$};

\path [draw = black,very thick,draw opacity = 1] (0.5,0.866025) --
(1,0) node[black,below] {$H_{\theta,1}$};

\path [draw = black, very thick, draw opacity = 1, shift = {(2,0)}]
(0,0) -- +(60:6) node[black,above] {$H_{\alpha_1,2}$};

\path [draw = black, very thick, draw opacity = 1] (1,1.73205) -- +(6,0);
\draw [shift = {(1,1.73205)}] (6.5,0) node {$H_{\alpha_2,2}$};

\path [draw = black,very thick,draw opacity = 1] (1,1.73205) -- (2,0) node[black,below] {$H_{\theta,2}$};

\path [draw = black,dashed,very thick,draw opacity = 1] (1.5,1.5*\a) -- (3,0) node[black,below] {$H_{\theta,3}$};

\path [draw = black,dashed,very thick,draw opacity = 1] (2,2*\a) -- (4,0) node[black,below] {$H_{\theta,4}$};



\end{scope}

\end{tikzpicture}\\
\caption{\small $\shi{}{3}{2}$ consists of 12 regions}
\label{fig:alcoveRegionExample}
\end{center}
\end{figure}
\end{example}
Denote the {\it Shi tableau} for the alcove $\A$ by $T_{\A}$ and for
the region $\region{}$ by $T_{\region{}}$.

Both \cite{Richards} and \cite{A2005a}
characterized the Shi tableaux for dominant $\m$-Shi regions.

\begin{lemma}
\label{lemma:shiTabReg}
Let $T=\{e_{ij}\}_{1\leq i\leq j\leq n-1}$ be a collection of
  integers such that $0\leq e_{ij}\leq m$. Then $T$ is the Shi tableau
  for a region $\region{}\in\shi{}{n}{m}$ if and only if
  \begin{equation}
\label{eqn:shiTabReg}
e_{ij}=\begin{cases}
e_{it}+e_{t+1,j}\text{ or }e_{it}+e_{t+1,j}+1&\text{if
}m-1\geq e_{it}+e_{t+1,j}\text{ for }t=i,\ldots,j-1\\
m&\text{otherwise}\\
\end{cases}
\end{equation}
\end{lemma}
\begin{proof}
  \cite{A2005a} defined co-filtered chains of ideals as decreasing
  chains of ideals in the root poset $$\posRoots=I_0\supseteq
  I_1\supseteq \ldots\supseteq I_m$$ in $\posRoots$ such that
\begin{equation}
\label{eqn:cofilterI}
(I_i+I_j)\cap\posRoots\subseteq I_{i+j},
\end{equation}
and
\begin{equation}
\label{eqn:cofilterJ}
(J_i+J_j)\cap\posRoots\subseteq J_{i+j},
\end{equation}
where $I_k=I_m$ for $k>m$ and $J_i=\posRoots\setminus I_i.$ He gave a
bijection between co-filtered chains of ideals and $m$-minimal alcoves
for $\region{}\in\shi{}{n}{m}$. Given such a chain, let $e_{uv}=k$ if
$\a_{uv}\in I_k$, $\a_{uv}\notin I_{k+1}$, and $k<m$ and let
$e_{uv}=m$ if $\a_{uv}\in I_m$. Then conditions (\ref{eqn:cofilterI})
and (\ref{eqn:cofilterJ}) translate into (\ref{eqn:shiTabReg}).
\end{proof}

Lemma 3.9 from \cite{A2005a} is crucial to our work here.  He
characterizes the co-filtered chains of ideals for which $\Hak{\a}{m}$
is a separating wall. We translate that into our set-up in
Lemma~\ref{lem:indecomp}, using entries from the Shi
Tableau.

\begin{lemma}[\cite{A2005a}]
\label{lem:indecomp}
A region $\region{}\in\shi{}{n}{m}$ has
$\HH{u}{v}{m}$ as a separating wall
if and only if $e_{uv}=m$ and for all $t$
such that $u\leq t<v$, $e_{ut}+e_{t+1,v}=m-1$.
\end{lemma}

\subsection{Partitions}
\label{sec:partitions}
A {\it partition} is a non-increasing sequence
$\lambda=(\lambda_1,\lambda_2,\ldots,\lambda_n)$ of nonnegative
integers, called the {\it parts} of $\lambda$. We identify a partition
$\lambda=(\lambda_1,\lambda_2,\ldots, \lambda_n)$ with its Young
diagram, that is the array of boxes with coordinates $\{ (i,j): 1 \leq
j \leq \lambda_i \mbox{ for all } \lambda_i \} .$ The {\it conjugate}
of $\lambda$ is the partition $\lambda'$ whose diagram is obtained by
reflecting $\lambda$'s diagram about the diagonal. The {\it length} of
a partition $\lambda$, $\length(\lambda)$, is the number of positive
parts of $\lambda$.

\subsubsection{Core partitions}
The {\it $(k,l)$-hook} of any partition $\lambda$ consists of the
$(k,l)$-box of $\lambda$, all the boxes to the right of it in row $k$
together with all the boxes below it and in column $l$. The {\it hook
length\/} $h_{kl}^\lambda$ of the box $(k,l)$ is the number of boxes
in the $(k,l)$-hook.  Let $n$ be a positive integer. An {\it
$n$-core\/} is a partition $\lambda$ such that $n \nmid
h_{(k,l)}^\lambda$ for all $(k,l) \in \lambda$.  We let $\core{n}$
denote the set of partitions which are $n$-cores.

\subsubsection{$\affS$ action on cores}
\label{subsec:affCores}
There is a well-known action of $\affS$ on $n$-cores which we will
briefly describe here; please see \cite{MisraMiwa}, \cite{Lascoux01},
\cite{LapointeMorse}, \cite{BJV}, or \cite{FV1}, for more details and
history.

The Young diagram of a partition $\lambda$ is made up of boxes. We
say the box in row $i$ and column $j$ has \textit{residue $r$} if
$j-i\equiv r\mod n$. A box not in the Young diagram of $\lambda$ is
called {\em addable} if we obtain a partition when we add it to
$\lambda$. In other words, the box $(i,j+1)$ is addable if
$\lambda_i=j$ and either $i=1$ or $\lambda_{i-1}>\lambda_i$. A box in
the Young diagram of $\lambda$ is called {\em removable} if we obtain
a partition when we remove it from $\lambda$. It is well-known (see
for example \cite{FV1} or \cite{LapointeMorse}) that the following
action of $s_i\in\affS$ on $n$-cores is well-defined.

\begin{definition}
$\affS$ action $n$-core partitions:
\begin{enumerate}
\item If $\lambda$ has an addable box with residue $r$, then
$s_r(\lambda)$ is the $n$-core partition created by adding all addable
boxes of residue $r$ to $\lambda$.

\item If $\lambda$ has an removable box with residue $r$, then $s_r(\lambda)$
is the $n$-core partition created by removing all removable boxes of residue $r$ from $\lambda$.

\item If $\lambda$ has neither removable nor addable boxes of residue $r$,
then $s_r(\lambda)$ is $\lambda$.
\end{enumerate}
\end{definition}

\subsection{Abacus diagrams}
\label{subsection;abacus-diagrams}
In Section~\ref{sec:firstSepWall}, we use a bijection, called
$\RiMap$, to describe certain regions. We will need abacus diagrams to
define $\RiMap$.  We associate to each partition
$\lambda$ its abacus diagram. When $\lambda$ is an $n$-core, its
abacus has a particularly nice form.

The {\it $\beta$-numbers} for a partition $\lambda = (\lambda_1,
\dots, \lambda_r)$ are the hook lengths from the boxes in its first
column: $$ \beta_k = h_{(k,1)}^{\lambda}.$$ Each partition is
determined by its $\beta$-numbers and $\beta_1>\beta_2>\cdots
>\beta_{\ell(\lambda)}>0$.

An \textit{$n$-abacus diagram}, or abacus diagram when $n$ is clear,
is a diagram with integer entries arranged in $\n$ columns labeled $0,
1, \dots, \n-1$. The columns are called \textit{runners}.  The
horizontal cross-sections or rows will be called \em levels \em and
runner $k$ contains the integer entry $q \n + r$ on level $q$ where
$-\infty < q < \infty$.  We draw the abacus so that each runner is
vertical, oriented with $-\infty$ at the top and $\infty$ at the
bottom, and we always put runner $0$ in the leftmost position,
increasing to runner $\n-1$ in the rightmost position.  Entries in the
abacus diagram may be circled; such circled elements are called
\textit{beads}. The {\it level} of a bead labeled by $qn+r$ is $q$ and
its runner is $r$. Entries which are not circled will be called
\textit{gaps}.  Two abacus diagrams are equivalent if one can be
obtained by adding a constant to each entry of the other.

See Example \ref{ex-abacus} below.

Given a partition $\lambda$ its abacus is any abacus diagram
equivalent to the one with beads at entries $\beta_k =
h_{(k,1)}^{\lambda}$ and all entries $j \in \ZZ_{< 0}$.

Given the original $n$-abacus for the partition $\lambda$ with beads at
$\{\beta_k\}_{1\leq k\leq\ell(\lambda)}$, let $b_i$ be one more than
the largest level number of a bead on runner $i$; that is, the level
of the first gap. Then $(b_0,\ldots,b_{n-1})$ is the {\it vector of
level numbers} for $\lambda$.

The \textit{balance number} of an abacus is the sum over all runners
of the largest level of a bead in that runner.  An abacus
is \textit{balanced} if its balance number is zero.  There
is a unique $n$-abacus which represents a given $\n$-core $\lambda$ for
each balance number.  In particular, there is a unique $n$-abacus for
$\lambda$ with balance number $0$.

\begin{remark}
  It is well-known that $\lambda$ is an $\n$-core if and only if all
  its $n$-abacus diagrams are \textit{flush}, that is to say whenever
  there is a bead at entry $j$ there is also a bead at $j -
  \n$. Additionally, if $(b_0,\ldots,b_{n-1})$ is the vector of level
  numbers for $\lambda$, then $b_0=0$,
  $\sum_{i=0}^{n-1}b_i=\ell(\lambda)$, and since there are no gaps,
  $(b_0\ldots,b_{n-1})$ describes $\lambda$ completely.
\end{remark}

\begin{example}
\label{ex-abacus}
Both abacus diagrams in Figure~\ref{fig;abacus} represent the 4-core
$\lambda =(5,2,1,1,1)$.  The levels are indicated to the left of the
abacus and below each runner is the largest level number of a bead in
that runner. The boxes of the Young diagram of $\lambda$ have been
filled with their hooklengths. The diagram on the left is
balanced. The diagram on the right is the original diagram, where the
beads are placed at the $\beta$-numbers and negative integers. The
vector of level numbers for $\lambda$ is $(0,3,1,1)$.
\begin{figure}[ht]
\begin{center}
\begin{tikzpicture}[fill opacity=1,every to/.style={draw,dotted}]

[fill opacity=.1]
\tikzstyle{every node}=[font=\footnotesize]
\path[draw=black, very thick,dashed,draw opacity = 1] (0,0) to (0,4);
\path[shift = {(.75,0)},draw=black, very thick,dashed,draw opacity = 1] (0,0) to (0,4);
\path[shift = {(1.5,0)},draw=black, very thick,dashed,draw opacity = 1] (0,0) to (0,4);
\path[shift = {(2.25,0)},,draw=black, very thick,dashed,draw opacity = 1] (0,0) to (0,4);

\node[color=red] at (0, -.5) {2};
\node[color=red] at (.75, -.5) {0};
\node[color=red] at (1.5, -.5) {0};
\node[color=red] at (2.25, -.5) {-2};
\fill[fill=white] (-.75,1) circle (.25) node {2};
\fill[fill=white] (-.75,1.5) circle (.25) node {1};
\fill[fill=white] (-.75,2) circle (.25) node {0};
\fill[fill=white] (-.75,2.5) circle (.25) node {-1};
\fill[fill=white] (-.75,3) circle (.25) node {-2};

\shade[shading=ball, ball color = Gainsboro] (0,1) circle (.25) node {8};
\shade[shading=ball, ball color = Gainsboro] (0,1.5) circle (.25) node
{4};
\shade[shading=ball, ball color = Gainsboro] (0,2) circle (.25) node {0};
\shade[shading=ball, ball color = Gainsboro] (0,2.5) circle (.25) node {-4};
\shade[shading=ball, ball color = Gainsboro] (0,3) circle (.25) node {-8};
\fill[fill=white,fill opacity=.1] (.75,1) circle (.25) node {9};
\fill[fill=white] (.75,1.5) circle (.25) node
{5};
\shade[shading=ball, ball color = Gainsboro] (.75,2) circle (.25) node {1};
\shade[shading=ball, ball color = Gainsboro] (.75,2.5) circle (.25) node {-3};
\shade[shading=ball, ball color = Gainsboro] (.75,3) circle (.25) node {-7};

\fill[fill=white] (1.5,1) circle (.25) node {10};
\fill[fill=white] (1.5,1.5) circle (.25) node {6};
\shade[shading=ball, ball color = Gainsboro] (1.5,2) circle (.25) node {2};
\shade[shading=ball, ball color = Gainsboro] (1.5,2.5) circle (.25) node {-2};
\shade[shading=ball, ball color = Gainsboro] (1.5,3) circle (.25) node {-6};

\fill[fill=white] (2.25,1) circle (.25) node {11};
\fill[fill=white] (2.25,1.5) circle (.25) node
{7};
\fill[fill=white] (2.25,2) circle (.25) node {3};
\fill[fill=white] (2.25,2.5) circle (.25) node {-1};
\shade[shading=ball, ball color = Gainsboro] (2.25,3) circle (.25) node {-5};

\begin{scope}[shift={(4,0)}]
\node[color=red] at (0, -.5) {-1};
\node[color=red] at (.75, -.5) {2};
\node[color=red] at (1.5, -.5) {0};
\node[color=red] at (2.25, -.5) {0};
\fill[fill=white] (-.75,1) circle (.25) node {2};
\fill[fill=white] (-.75,1.5) circle (.25) node {1};
\fill[fill=white] (-.75,2) circle (.25) node {0};
\fill[fill=white] (-.75,2.5) circle (.25) node {-1};
\fill[fill=white] (-.75,3) circle (.25) node {-2};

\path [draw=black, very thick,dashed,draw opacity = 1](0,0) to (0,4);
\path[shift = {(.75,0)},draw=black, very thick,dashed,draw opacity = 1] (0,0) to (0,4);
\path[shift = {(1.5,0)},draw=black, very thick,dashed,draw opacity = 1] (0,0) to (0,4);
\path[shift = {(2.25,0)},draw=black, very thick,dashed,draw opacity = 1] (0,0) to (0,4);
\shade[shading=ball, ball color = Gainsboro] (.75,1) circle (.25) node {9};
\shade[shading=ball, ball color = Gainsboro] (.75,1.5) circle (.25) node{5};
\shade[shading=ball, ball color = Gainsboro] (.75,2) circle (.25) node {1};
\shade[shading=ball, ball color = Gainsboro] (.75,2.5) circle (.25) node {-3};
\shade[shading=ball, ball color = Gainsboro] (.75,3) circle (.25) node {-7};
\fill[fill=white] (1.5,1) circle (.25) node {10};
\fill[fill=white] (1.5,1.5) circle (.25) node{6};
\shade[shading=ball, ball color = Gainsboro] (1.5,2) circle (.25) node {2};
\shade[shading=ball, ball color = Gainsboro] (1.5,2.5) circle (.25) node {-2};
\shade[shading=ball, ball color = Gainsboro] (1.5,3) circle (.25) node {-6};

\fill[fill=white] (2.25,1) circle (.25) node {11};
\fill[fill=white] (2.25,1.5) circle (.25) node {7};
\shade[shading=ball, ball color = Gainsboro] (2.25,2) circle (.25) node {3};
\shade[shading=ball, ball color = Gainsboro] (2.25,2.5) circle (.25) node {-1};
\shade[shading=ball, ball color = Gainsboro] (2.25,3) circle (.25) node {-5};

\fill[fill=white] (0,1) circle (.25) node {8};
\fill[fill=white](0,1.5) circle (.25) node {4};
\fill[fill=white] (0,2) circle (.25) node {0};
 \shade[shading=ball, ball color = Gainsboro] (0,2.5) circle (.25) node {-4};
\shade[shading=ball, ball color = Gainsboro] (0,3) circle (.25) node {-8};
\end{scope}
\begin{scope}[shift={(3,-.75)}]
\node[shift = {(-1,-1)}] {$\lambda =$};
\def\boxpath{-- +(-0.5,0) -- +(-0.5,-0.5) -- +(0,-0.5) -- cycle};
\draw[step =0.5,shift={(0,0)}, thick](0,-0.5) \boxpath node[anchor= south east ]{9};
\draw[step =0.5,shift={(0,0)}, thick](0.5,-0.5) \boxpath node[anchor= south east ]{5};
\draw[step =0.5,shift={(0,0)}, thick](1,-0.5) \boxpath node[anchor= south east ]{3};
\draw[step =0.5,shift={(0,0)}, thick](1.5,-0.5) \boxpath node[anchor= south east ]{2};
\draw[step =0.5,shift={(0,0)}, thick](2,-0.5) \boxpath node[anchor= south east ]{1};
\draw[step =0.5,shift={(0,0)}, thick](0,-1) \boxpath node[anchor= south east ]{5};
\draw[step =0.5,shift={(0,0)}, thick](0.5,-1) \boxpath node[anchor= south east ]{1};
\draw[step =0.5,shift={(0,0)}, thick](0,-1.5) \boxpath node[anchor= south east ]{3};
\draw[step =0.5,shift={(0,0)}, thick](0,-2) \boxpath node[anchor= south east ]{2};
\draw[step =0.5,shift={(0,0)}, thick](0,-2.5) \boxpath node[anchor= south east ]{1};

\end{scope}
\end{tikzpicture}

\caption{\small     The abacus represents the $4$-core $\lambda$.  }
  \label{fig;abacus}
\end{center}
\end{figure}
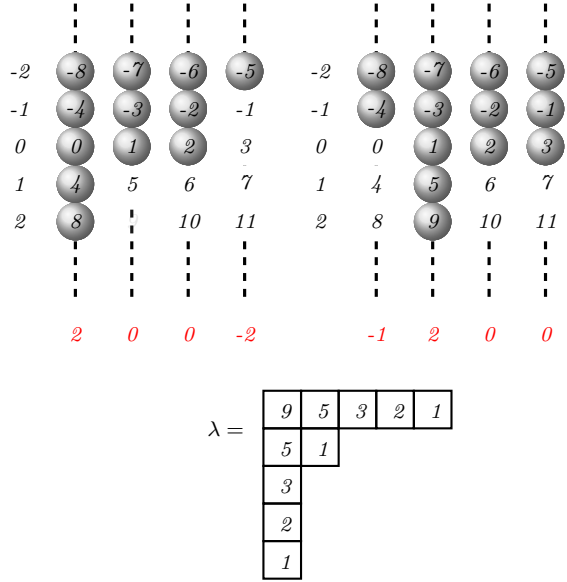

\end{example}

\subsection{Bijections}
\label{sec:FVBij}
We describe here two bijections, $\RiMap$ and $\FVMap$, from the set
of $n$-cores to dominant alcoves. We neither use nor prove the fact
that $\RiMap=\FVMap$.

\subsubsection{Combinatorial description}
\label{subsec:FVBij1}
$\RiMap$ is a slightly modified version of the bijection
given in \cite{Richards}. Given an $n$-core $\lambda$, let
$(b_0=0,b_1,\ldots,b_{n-1})$ be the level numbers for its abacus. Now
let $\tilde{p}_i=b_{i-1}n+i-1$, which is the entry of the first gap on
runner $i$, for $i$ from $1$ to $n$, and then let
$p_1=0<p_2<\cdots<p_n$ be the $\{\tilde{p}_i\}$ written in ascending
order. Finally we define $\RiMap(\lambda)$ to be the alcove whose Shi
coordinates are given by
$$k_{ij}=\lfloor\frac{p_{j+1}-p_i}{n}\rfloor$$ for $1\leq i\leq j\leq
n-1$.

\begin{example}
\label{ex:fvmap}
We continue Example~\ref{ex-abacus}. We have $n=4$,
$\lambda=(5,2,1,1,1)$, and $(b_0,b_1,b_2,b_3)=(0,3,1,1)$. Then
$\tp_1=0$, $\tp_2=13$, $\tp_3=6$, and $\tp_4=7$ and $p_1=0$, $p_2=6$,
$p_3=7$, and $p_4=13$. Thus $\RiMap(\lambda)$ is the alcove with the
following Shi tableau.\\

\begin{tikzpicture}[font=\footnotesize]
\youngDiagram{{{$k_{13}=3$,$k_{12}=1$,$k_{11}=1$},{$k_{23}=1$,$k_{22}=0$},$k_{33}=1$}}{1.1}
\end{tikzpicture}

\end{example}

\begin{proposition}
\label{prop:FVBij}
The map $\RiMap$ from $n$-cores to dominant alcoves is a bijection.
\end{proposition}
\begin{proof}
  We first show that we indeed produce an alcove by the
  process above. By Lemma~\ref{lem:alcoveCoords}, it is enough to show
  that $k_{it}+k_{t+1,j}\leq k_{ij}\leq k_{it}+k_{t+1,j}+1$ for all $t$
  such that $1\leq t<j$.

$k_{ij}=\lfloor\frac{p_{j+1}-p_i}{n}\rfloor$ implies that
\begin{equation}
\label{eqn:altK}
k_{ij}=\frac{p_{j+1}-p_i}{n}-B_{ij}\textrm{ where }0\leq B_{ij}<1.
\end{equation}
 Let $t$ be such that $0\leq t<j$. Using (\ref{eqn:altK}), we have
$$k_{it}+k_{t+1,j}=\frac{p_{t+1}-p_i}{n}-\frac{p_{j+1}-p_{t+1}}{n}-B_{ij}-B_{j+1,t+1}.$$
Now let $A=B_{ij}+B_{j+1,t+1}$, so that $0\leq A<2$. We have
$$k_{it}+k_{t+1,j}+A=\frac{p_{j+1}-p_i}{n}.$$ Thus $$\lfloor
k_{it}+k_{t+1,j}+A\rfloor = \lfloor
\frac{p_{j+1}-p_i}{n}\rfloor=k_{ij}$$ or, since $k_{it}$ and
$k_{t+1,j}$ are integers,
\begin{equation}
\label{eqn:ksum}
k_{it}+k_{t+1,j}+\lfloor A\rfloor = k_{ij}.
\end{equation} Combining (\ref{eqn:ksum}) with $\lfloor
A\rfloor$ is equal to $0$ or $1$ shows that the conditions in
Lemma~\ref{lem:alcoveCoords} are satisfied and we have the Shi
coordinates of an alcove. Since each $k_{ij}\geq 0$, it is an alcove
in the dominant chamber.

Now we reverse the process described above to show that $\RiMap$ is a
bijection. Let $\{k_{ij}\}_{1\leq i\leq j\leq n-1}$ be the Shi
coordinates of a dominant alcove.  Write $p_i=nq_i+r_i$ for the
intermediate values $\{p_i\}$, which we first calculate. Then
$p_1=q_1=r_1=0$ and $q_i=k_{1,i-1}$. We must now determine
$r_2,\ldots,r_n$, a permutation of $1,\ldots,n-1$. However, since
\begin{equation}
\label{eqn:kij}
k_{ij}=\begin{cases}q_{j+1}-q_{i}&\text{if }r_{j+1}>r_{i}\\
q_{j+1}-q_{i}-1&\text{if }r_{j+1}<r_{i}.
\end{cases},
\end{equation}

 we can determine the inversion table for this permutation, using
$k_{ij}$ for $2\leq i\leq j\leq n-1$ and $q_1,\ldots,q_n$. Indeed,
\begin{align}
\label{eqn:inv}
\Inv(r_{j+1})&=|\{r_i\mid 1\leq i<j+1\textrm{ and }r_i>r_{j+1}\}|\nonumber\\
&=|\{(k_{1j},k_{1,{i-1}},k_{ij})\mid k_{1j}=k_{1,i-1}+k_{ij}+1\}|.
\end{align}

Therefore, we can compute $r_2,\ldots,r_n$ and therefore
 $p_1,p_2,\ldots,p_n$. We can now sort the $\{p_i\}$ according to
 their residue mod $n$, giving us $\tilde{p}_1,\ldots,\tilde{p}_n$;
 from this, $(b_0,\ldots,b_{n-1})$. Note that $(b_0,\ldots,b_{n-1})$
 is a permutation of $q_1,\ldots,q_n$.
\end{proof}

\begin{example}
We continue Examples~\ref{ex-abacus} and \ref{ex:fvmap} here. Suppose
we are given that $n=4$ and the alcove coordinates
$k_{13}=3$,$k_{12}=1$,$k_{11}=1$,$k_{23}=1$,$k_{22}=0$, and
$k_{33}=1$. That is,

\begin{tikzpicture}
\node at (-.6,0) {$T_{\region{}}=$};
\youngDiagram{{{$k_{13}$, $k_{12}$,
    $k_{11}$}, { $k_{23}$,
    $k_{22}$},{$k_{33}$}}}{.6}
\begin{scope}[shift={(3,0)}]
\node at (-.6,0) {$=$};
\youngDiagram{{{$3$, $1$,
    $1$}, {$1$,
    $0$},{$1$}}}{.6}
\end{scope}

\end{tikzpicture}

We demonstrate $\RiMap^{-1}$ and calculate
$(b_0,b_1,b_2,b_3)$ and thereby the $4$-core $\lambda$.  We have
$q_1=0$, $q_2=1$, $q_3=1$, and $q_4=3$, and $r_1=0$, from $k_{13}$,
$k_{12}$, and $k_{11}$. We must determine $r_2$, $r_3$, $r_4$, a
permutation of $1,2,3$.
\smallskip

Using (\ref{eqn:inv}), we know $\Inv(r_4)=2$, since $k_{13}=k_{11}+k_{23}+1$ and $k_{13}=k_{12}+k_{33}+1$.
\smallskip

$\Inv(r_3)=0$, since $k_{12}\neq k_{11}+k_{22}+1$.
\smallskip

$\Inv(r_2)=0$, always.
\smallskip

Therefore we have $r_3=3$, $r_2=2$, and $r_4=1$, which means
$b_1=q_4=3$, $b_2=q_2=1$, and $b_3=q_3=1$.
\end{example}

\begin{remark}
  The column (or row) sums of the Shi tableau of an alcove give us a
  partition whose conjugate is $(n-1)$-bounded, as in the bijections
  of \cite{LapointeMorse} or \cite{BB1996}
\end{remark}

\subsubsection{Geometric description}
\label{subsec:FVBij2}
The bijection $\FVMap$ associates an $n$-core to an alcove through the
$\affS$ action described in Sections~\ref{subsec:affCores} and
\ref{sec:affAlcoves}. The map $\FVMap:w\emptyset\mapsto w^{-1}\Afund$
for $w\in\affS$ a minimal length coset for $\affS/\Sn$, is a
bijection. In \cite{FV1}, it is shown that the $m$-minimal alcoves of
Shi regions in $\shinm$ correspond, under $\FVMap$, to $n$-cores which
are also $(nm+1)$-cores.

\section{Separating wall $\Hthm$}
\label{sec:firstSepWall}
Separating walls were defined in Section~\ref{subsec:shi} as a wall of
a region which separates the region from $\Afund.$ Equivalently for
alcoves, $\Hak{\a}{k}$ is a separating wall for the alcove
$w^{-1}\Afund$ if there is a simple reflection $s_i$, where $0\leq
i<n$, such that $w^{-1} \Afund \subseteq \Hp{\alpha}{k}\text{ and
}(s_iw)^{-1} \Afund \subseteq \Hn{\alpha}{k}$.  We want to count the
regions which have $\Halpha{}{}{m}$ as a separating wall, for any
$\a\in\posRoots$. We do this by induction and the base case will be
$\a=\theta.$ Our main result in this section characterizes the regions
which have $\Hthm$ as a separating wall by describing the $n$-core
partitions associated to them under the bijections $\RiMap$ and
$\FVMap$ described in Section~\ref{sec:FVBij}.

\begin{theorem}
\label{thm:baseCase}
Let $\RiMap:\core{n}\to\domAlcoves_n$ be the bijection described in
Section~\ref{subsec:FVBij1}, let $\region{}\in\shinm$ have $m$-minimal
alcove $\A$, and let $\lambda$ be the $n$-core such that
$\RiMapArg{\lambda{}}=\A$.  Then $\Hthm$ is a separating wall for the
region $\region{}$ if and only if $h_{11}^{\lambda}=n(m-1)+1$.
\end{theorem}

\begin{proof}

Let $\vec b(\lambda)=(b_0,b_1,\ldots,b_{n-1})$ be the vector of level
numbers for the $n$-core $\lambda$, so $b_o=0$. We first note that
$h_{11}=\beta_1=n(m-1)+1$ if and only if $b_1=m$ and $b_i<m$ for $1<i\leq
n-1$.

\medskip

Now suppose that $\Hthm$ is a separating wall for the region
$\region{}$. Let $\{e_{ij}\}$ be the coordinates of $\region{}$ and
let $\{k_{ij}\}$ be the coordinates of $\A$. By
Lemma~\ref{lem:indecomp}, we know that $e_{1,n-1}=m$ and
$e_{1t}+e_{t+1,n-1}=m-1$, for all $t$ such that $1\leq
t<n-1$. Therefore for all $e_{ij}$ except $e_{1,n-1}$, we have
$e_{ij}\leq m - 1$, so that $e_{ij}=k_{ij}$. Since
$k_{1t}+k_{t+1,n-1}\leq k_{1,n-1}\leq k_{1t}+k_{t+1,n-1}+1$, we have
that $k_{1,n-1}\leq m$, so indeed the Shi coordinates of $\region{}$
are the same as the coordinates of $\A$.

Consider the proof of Proposition~\ref{prop:FVBij} where we describe
$\RiMap^{-1}$, but in this situation.  We see that $\{q_i\}_{1\leq
  i\leq n-1}$, a nonincreasing rearrangement of $(b_1\ldots,b_{n-1})$,
is made up of $m$ and $n-2$ nonnegative integers strictly less than
$m$. So we need only show that $b_1=m$, in view of our first remark of
the proof.  Combining (\ref{eqn:inv}) with the facts that if $\Hthm$
is a separating wall for a region then $e_{ij}=k_{ij}$ and, then by
Lemma~\ref{lem:indecomp}, $k_{1,n-1}=k_{1,i-1}+k_{i,n+1}+1$ for all
$i$ such that $2\leq i\leq n$, we have $\Inv(r_n)=n-1$. This implies
that $r_n=1$, so that $b_1=q_n=k_{1,n-1}=m$.

\sfOmitt{Set the intermediaries $p_i=q_in+r_i$ for $1\leq i\leq n$,
so that $p_1=r_1=q_1=0$ and $q_i=k_{1,i-1}$. Since $k_{1,n-1}=m$ and
$k_{1j}<m$ for $j<n-1$, we know $q_n=m$ and $q_j<m$ for $1\leq j\leq
n-1$. Thus $p_n=mn+r_n$. Thus to show that $b_1=m$, we must show that
$r_n=1$, meaning $p_n=\tilde{p}_1$.

Since $r_1,\ldots,r_n$ are pairwise distinct positive integers, we
know $r_n=1$ if and only if $r_n<r_i$ for $2\leq i\leq n-1$. By
specializing (\ref{eqn:kij}) to $j=n-1$, we have
\begin{equation}
\label{eqn:kin_1}
k_{i,n-1}=\begin{cases}q_{n}-q_{i}&\text{if }r_{n}>r_{i}\\
q_{n}-q_{i}-1&\text{if }r_{n}<r_{i}.
\end{cases}.
\end{equation}

By Lemma~\ref{lem:indecomp}, we know that
\begin{equation}
\label{eqn:k1n_1}
k_{1,i-1}+k_{i,n-1}=m-1
\end{equation}

Combining (\ref{eqn:kin_1}), (\ref{eqn:k1n_1}), and $k_{1,i-1}=q_i$,
we see that indeed $r_n<r_i$ for $2\leq i\leq n-1$.}

\medskip
Conversely, suppose that $h^{\lambda}_{11}=n(m-1)+1$, so that $b_1=m$
and $b_i\leq m-1$ for $1<i\leq n-1$. Then $\tilde{p}_2=nm+1$ and
$\tilde{p}_i=nb_{i-1}+i-1 \leq n(m-1)+i-1\leq
n(m-1)+n-1=nm-1$. Therefore, $p_1=0$ and $p_n=nm+1$ and $p_i\leq
nm-1$, so that $q_1=0$, $q_n=m$, $r_n=1$, and $q_i\leq m-1$ and thus
$k_{1,n-1}=m$ and $k_{1i}\leq m-1$.  By specializing (\ref{eqn:kij})
to $j=n-1$, we have
\begin{equation}
\label{eqn:kin_1}
k_{i,n-1}=\begin{cases}q_{n}-q_{i}&\text{if }r_{n}>r_{i}\\
q_{n}-q_{i}-1&\text{if }r_{n}<r_{i}.
\end{cases}.
\end{equation}
Then, by (\ref{eqn:kin_1}), $k_{i,n-1}=q_n-q_i-1$, so that
$$k_{1,i-1}+k_{i,n-1}=q_i+q_n-q_i-1=m-1.$$ Since $k_{ij}\leq m$ for
$1\leq i\leq j\leq n-1$, $k_{ij}=e_{ij}$ and the conditions in
Lemma~\ref{lem:indecomp} that $\Hthm$ be a separating wall are
fulfilled.
\end{proof}

We can also look at the regions which have $\Hthm$ as a separating
wall in terms of the geometry directly. Theorem~\ref{thm:altBaseCase}
is an alternate version of Theorem~\ref{thm:baseCase}.
\begin{proposition}
\label{prop-gamma}
Let $\lambda$ be an $n$-core and $w \in \coset$
be of minimal length such that
$\lambda = w \emptyset$.  Let $k = \lambda_1 + \frac{n-1}{2}$.
 Let $\gamma =  \alij{1}{n-1} +  \alij{2}{n-1} + \cdots + \alij{n-1}{n-1}$.
\begin{enumerate}
\item
Then the affine hyperplane $\Hak{\gamma}{k}$ passes through the
corresponding alcove $w^{-1}\Afund$.  More precisely,
$\brac{w^{-1}(\frac 1n \rho)}{\gamma} = k$.
\item
Then the affine hyperplane $\Hak{\gamma}{\lambda_1 }$ passes through
the corresponding alcove $w^{-1}\Afund$.  More precisely,
$\brac{w^{-1}(\Lambda_{r})}{\gamma} = \lambda_1 $, where $r \equiv
\lambda_1 \mod n$.
\end{enumerate}
\end{proposition}
\begin{proof}
First, recall that $\rho = \frac 12 \sum_{\a \in \posRoots} \a =
(\frac{n-1}{2}, \frac{n-1}{2}-1, \ldots, \frac{1-n}{2})$.

Hence $\frac 1n \rho \in \Afund$ and so $w^{-1}(\frac 1n \rho) \in
w^{-1}\Afund$.  Let $\eta = \sum_i \ei$. Recall $V = \eta^\perp$ as
for all $(a_1, \ldots, a_n) \in V$ we have $\sum_i a_i = 0$. Observe
that for all $v \in V$, $\brac{v}{\gamma} = \brac{v}{\eta- n \e_n}
= \brac{v}{-n \e_n}$.
So it suffices to show $\brac{w^{-1}(\frac 1n \rho)}{\e_n} = -\frac kn$.

Recall we may write $w = t_\beta u$ where $\beta \in Q$ and $u \in
 \Sn$, where $t_\beta$ is translation by $\beta$. Please see
 \cite{hum90} for details.  Then $w^{-1} = u^{-1} t_{-\beta} =
 t_{u^{-1}(-\beta)} u^{-1}$ satisfies $u^{-1} (-\beta) \in Q^+$.

Write $\lambda_1 = nq - (n-r)$ with $0 \le n-r < n$.  Then $1 \le r
\le n$, $q = \ceil{\lambda_1}{n}$, and $r \equiv \lambda_1 \mod n$.
Let $a_i$ be the level of the first gap in runner $i$ of the balanced
abacus diagram for $\lambda$ and write $\vec
n(\lambda)=(a_1,a_2,\ldots,a_n)$. It is worth noting that $\vec
n(\lambda)=w(0,\ldots,0)$.  By \cite[Prop 3.2.13]{BJV}, the largest
entry of $\vec n(\lambda)$ is $a_r = q$ and the rightmost occurrence
of $q$ occurs in the $r^{\mathrm{th}}$ position.  Hence the smallest
entry of $-\vec n(\lambda)$ is $-q$ and its rightmost occurrence is
also in position $r$.  Since $u^{-1} \in \Sn$ is of minimal length
such that $u^{-1} (-\vec n(\lambda)) \in Q^+$, we have that
$u^{-1}(\e_r) = \e_n$.

Now we compute
\begin{eqnarray*}
\brac{w^{-1}(\frac 1n \rho)}{\e_n} &=&
		\brac{t_{u^{-1}(-\vec n(\lambda))} u^{-1} (\frac 1n \rho)}{\e_n} \\
&=&
\brac{u^{-1}(-\vec n(\lambda))}{\e_n}  + \brac{u^{-1} (\frac 1n \rho)}{\e_n}
\\
&=&
 \brac{u^{-1}(-\vec n(\lambda))}{\e_n}  + \brac{ \frac 1n \rho}{u(\e_n)}
\\
&=&
-q + \brac{ \frac 1n \rho}{\e_r}
	= -q +  \frac 1n(\frac{n-1}{2} - (r-1))
\\
&=&
  -\frac 1n(nq - (n-r) + \frac{n-1}{2})
  	 = -\frac 1n(\lambda_1 + \frac{n-1}{2})
\\
&=&
 -\frac kn.
\end{eqnarray*}

\medskip

For the second statement, note the fundamental weight $\Lambda_j \in V$ has
coordinates given by
$$\Lambda_j = \frac 1n ( (n-j)(\e_1 + \cdots + \e_j) - j(\e_{j+1} + \cdots + \e_n)).$$

So $\Lambda_j \in \Hak{\a_i}{0}$ for $i \neq j$, $\Lambda_j \in
\Hak{\a_j}{1}$, and the $\{\Lambda_j \mid 1 \le j \le n\} \cup \{0\}$
are precisely the vertices of $\Afund$.  For the notational
consistency of this statement and others below, we will adopt the
convention that $\Lambda_0 = 0$ (which is consistent with considering
$0 \in \Hak{\theta}{0} = \Hak{\alpha_0}{1}$).  Hence we have that
$w^{-1}(\Lambda_j) \in w^{-1}\Afund$.

As above we compute
\begin{eqnarray*}
- \frac 1n \brac{w^{-1}(\Lambda_{r})}{\gamma} &=&
\brac{w^{-1}(\Lambda_{r})}{\e_n}
\\  &=&
-q + \brac{ \Lambda_{r}}{\e_r}
	= -q +  \frac 1n(n - r)
\\
&=&
  -\frac 1n(nq -(n-r))
  	 = -\frac 1n(\lambda_1).
\end{eqnarray*}

\end{proof}
%
\begin{proposition}
\label{prop-Gamma}
Let $\lambda$ be an $n$-core and $w \in \coset$
be of minimal length such that
$\lambda = w \emptyset$.  Let $K = \ell(\lambda) + \frac{n-1}{2}$.
 Let $\Gamma =  \alij{1}{n-1} +  \alij{1}{n-2} + \cdots + \alij{1}{1}$.
\begin{enumerate}
\item
Then the affine hyperplane $\Hak{\Gamma}{K}$ passes through the corresponding
alcove $w^{-1}\Afund$.
More precisely, $\brac{w^{-1}(\frac 1n \rho)}{\Gamma} = K$.
\item
Then the affine hyperplane $\Hak{\Gamma}{\ell(\lambda)}$ passes through the corresponding
alcove $w^{-1}\Afund$.
More precisely, $\brac{w^{-1}(\Lambda_{s-1})}{\Gamma} = \ell(\lambda)$,
where $1-s \equiv \ell(\lambda) \mod n$.
\end{enumerate}
\end{proposition}
\begin{proof}
First note $\brac{v}{\Gamma} = \brac{v}{n \e_1}$ for all $v \in
V$, so it suffices to compute $\brac{w^{-1}(\frac 1n \rho)}{n \e_1}$.

Next, note $\Gamma = n \e_1$.

Write  $\ell(\lambda)  = nM + (1-s)$ with $1 \le s \le n$, so
$-M =- \ceil{\ell(\lambda)}{n}$.
By \cite{BJV}, the smallest entry of $\vec n(\lambda) = (a_1, a_2, \ldots, a_n)$
is $a_s = -M$ and the leftmost occurrence of $-M$ occurs in the $s^{\mathrm{th}}$ position.
Hence the largest entry of $-\vec n(\lambda)$ is $M$ and its leftmost occurrence is
also in position $s$.   Then for $u$ as above, it is clear $u(\e_1) = \e_s$.
So, by a similar computation as above,
\begin{eqnarray*}
\brac{w^{-1}(\frac 1n \rho)}{\Gamma}   &=&
		n\brac{w^{-1}(\frac 1n \rho)}{\e_1}
 \\
&=&
n\brac{u^{-1}(-\vec n(\lambda))}{\e_1}  +n \brac{u^{-1} (\frac 1n \rho)}{\e_1}
\\
&=&
nM + \brac{ \rho}{\e_s}
	= n M +  \frac{n-1}{2} - (s-1)
\\
&=&
  	 \ell(\lambda)+ \frac{n-1}{2} = K.
\end{eqnarray*}

Likewise,
\begin{eqnarray*}
\brac{w^{-1}(\Lambda_{s-1})}{\e_1} &=&
M + \brac{ \Lambda_{s-1}}{\e_s}
	= M +  \frac 1n( - (s-1))
\\
&=&
  \frac 1n(nM + (1-s))
  	 = \frac 1n \ell(\lambda).
\end{eqnarray*}
\end{proof}
Taking subscripts $\mod n$ we have $\brac{w^{-1}(\Lambda_{\lambda_1})}{\gamma} = \lambda_1  $ and
 $\brac{w^{-1}(\Lambda_{-\ell(\lambda)})}{\Gamma} =  \ell(\lambda)$.
\begin{corollary}
\label{cor:hookInnerProd}
$n\brac{w^{-1}(\frac 1n \rho)}{\theta} =  \lambda_1 + \ell(\lambda) +n-1$.
\end{corollary}
Note that when $\lambda \neq \emptyset$, the above quantity is
 $h_{11}^{\lambda} +n$ where $h_{11}^{\lambda}$ is the
 hooklength of the first box.  (One could also set
 $h_{11}^{\emptyset} = -1$.)

\begin{theorem}
\label{thm:altBaseCase}
Let $\FVMap:\core{n}\to\domAlcoves_n$ be the bijection described in
Section~\ref{subsec:FVBij2}, let $\region{}\in\shinm$ have $m$-minimal
alcove $\A$, and let $\lambda$ be the $n$-core such that
$\FVMapArg{\lambda{}}=\A$.  Then $\Hthm$ is a separating wall for the
region $\region{}$ if and only if $h_{11}^{\lambda}=n(m-1)+1$.
\end{theorem}
\begin{proof}
  Let $r$, $s$, $q$, $M$, and $w$ be as in
  Propositions~\ref{prop-gamma} and \ref{prop-Gamma}.  Suppose that
  $\Hthm$ is a separating wall for $\region{}$ and let $i$ be such
  that $w^{-1}\Afund\subseteq\Hp{\theta}{m}$ and
  $w^{-1}s_i\Afund\subseteq\Hn{\theta}{m}$.  Recall $\Lambda_j \in
  \Hak{\ai}{0}$ for all $j \neq i$, and $\Lambda_i \in \Hak{\ai}{1}$.
  Hence $w^{-1}(\Lambda_j) \in \Hak{\theta}{m}$ but $w^{-1}(\Lambda_i)
  \in \Hak{\theta}{m+1}$.  In fact, this configuration of vertices
  characterizes separating walls.

Note
\begin{equation}
\label{eq-sr}
\brac{\Lambda_j}{\e_s - \e_r}=  \begin{cases} 1 & \text{ if } s \le j < r\\
		-1 & \text{ if } s > j \ge r \\
		0  & \text{else.} \end{cases}
\end{equation}
By Propositions \ref{prop-gamma} and \ref{prop-Gamma},
$\brac{w^{-1}(\Lambda_j)}{\theta} = M + q + \brac{\Lambda_j}{\e_s - \e_r}$.
Because $\Hak{\theta}{m}$ is a separating wall,
this yields $M + q + \brac{\Lambda_j}{\e_s - \e_r} = m + \delta_{i,j}$.
We must consider two cases.
First, $M+q=m$ and $\brac{\Lambda_j}{\e_s - \e_r} =  \delta_{i,j}$.
In other words,
by \eqref{eq-sr}
$s \le j < r$ implies $j=i$.
 More precisely,
$r-s = 1$,  $s=i$, and $\e_s - \e_r = \ai$.
In the second case,
 $M+q-1=m$ and $\brac{\Lambda_j}{\e_s - \e_r} =  \delta_{i,j} -1$.
In other words, $s > j \ge  r$  for all $1 \le j < n$ (and recall
$\brac{\Lambda_0}{\e_s - \e_r} = \brac{0}{\e_s - \e_r} = 0$).
 More precisely,
$r-s = 1- n$   and $\e_s - \e_r = -\theta$.

Putting this all together for $\lambda \neq \emptyset$,
\begin{eqnarray*}
h_{11}^\lambda
&=&
\ell(\lambda) + \lambda_1 -1 \\
&=&
(nM+ 1-s) + (n q - (n-r)) -1 \\
&=&
n(M+q) -n  + (r-s)
= \begin{cases} nm -n +1 & \text{if } \e_s - \e_r = \ai \\ n(m+1) -n +(1-n) & \text{if }  \e_s - \e_r = -\theta \end{cases} \\
&=&
n(m-1) + 1.
\end{eqnarray*}
\omitt{\mvNew{ cases labeled, ok ? also  did FTVbib fix
changed  germ to mathfrak instead}}

Conversely, if $h_{11}^{\lambda}=n(m-1)+1$, then by
the computation above $n(M+q) -n  + (r-s) = nm - n +1$, which forces
$n(M+q -1 -m +1) = 1 +s -r$.  Note $2-n \le 1+s-r \le n$.
  If $1+s-r < n$, divisibility forces $0 = 1+s-r = M+q-m$.
In other words, $\e_s - \e_r = \ai$ for $i=s$, and we compute as above
that $\brac{w^{-1}(\Lambda_j)}{\theta} = M + q + \delta_{i,j}$ showing
$\Hthm$ is a separating wall.
If instead $1+s-r = n$, this forces $M+q-m=1$ and $\e_s - \e_r = \theta$.
Hence $\brac{w^{-1}(\Lambda_j)}{\theta} = M + q-1 =m $ for all $j<n$,
but $\brac{w^{-1}(0)}{\theta} = M + q = m+1$,
 so that $\Hthm$ is again a separating wall for $w^{-1}\Afund$.

As a side note, similar calculations show that $h_{11}^\lambda = n(m-1) - 1$
if and only if either $M+q=m$ and $r-s = -1$, or $M+q=m-1 $ and $r-s = n-1$.
In both cases $\Hthm$ will not be
a separating  wall for $w^{-1}\Afund$, but will be
a separating wall for $w^{-1} s_i\Afund$ where $i=s-1$.
 One vertex of $w^{-1}\Afund$ lies in $\Hak{\theta}{m-1}$
and the rest in $\Hthm$.

\omitt{
Corollary~\ref{cor:hookInnerProd}, $\brac{w^{-1}(\frac 1n
\rho)}{\theta} = m+\frac{1}{n}$, so that $\Hthm$ must be a separating
wall.
}
\end{proof}

\section{Generating functions}
\label{sec:genfunc}
We use $\sepSet{n}{k}{\a}$ to denote the set of regions in $\shinm$
which have $\Halpha{}{}{k}$ as a separating wall. See
Figure~\ref{fig:alpha1Wall}. In the language of \cite{A2005a}, these
are the regions whose corresponding co-filtered chain of ideals have
$\a$ as an indecomposable element of rank $k$.

\begin{figure}[ht]
\begin{center}
\fpsacInclude{  \begin{tikzpicture}[scale=.7,font=\footnotesize]}
\fpsacOmitt{\begin{tikzpicture}[scale=.8,information text/.style={rounded corners,fill=red!10,inner sep=1ex}]}
\begin{scope}
\tikzstyle{every node}=[font=\footnotesize]
[4,1,2]

\def\a{1.73205}

\filldraw[fill = yellow, fill opacity = 1, draw opacity = 0] 
(2,0)--(5,3*\a)--(6,3*\a)--(6,0)--cycle;

\path [draw = black, very thick, draw opacity = 1, shift = {(0,0)}]
(0,0) -- +(60:6) node[black,above] {$H_{\alpha_1,0}$};

\path [draw = black,very thick, draw opacity = 1] (0,0) -- +(6,0) node[black,right] {$H_{\alpha_2,0}$};

\path [draw = red,very thick,draw opacity = 1] (0,0) -- (0,0);

\path [draw = black, very thick, draw opacity = 1, shift = {(1,0)}]
(0,0) -- +(60:6) node[black,above] {$H_{\alpha_1,1}$};

\path [draw = black,very thick, draw opacity = 1] (0.5,0.866025) -- +(6,0);
\draw [shift = {(0.5,0.866025)}] (6.5,0) node {$H_{\alpha_2,1}$};

\path [draw = black,very thick,draw opacity = 1] (0.5,0.866025) --
(1,0) node[black,below] {$H_{\alpha_{12},1}$};

\path [draw = red, very thick, draw opacity = 1, shift = {(2,0)}]
(0,0) -- +(60:6) node[black,above] {$H_{\alpha_1,2}$};

\path [draw = red, very thick, draw opacity = 1] (1,1.73205) -- +(6,0);
\draw [shift = {(1,1.73205)}] (6.5,0) node {$H_{\alpha_2,2}$};

\path [draw = red,very thick,draw opacity = 1] (1,1.73205) -- (2,0) node[black,below] {$H_{\alpha_{12},2}$};



\end{scope}
\end{tikzpicture}\\
\caption{\small There are three regions in $\sepSet{3}{2}{\a_1}$}
\label{fig:alpha1Wall}
\end{center}
\end{figure}

In this section, we present a generating function for regions in
$\sepSet{n}{k}{\a}$. In Section~\ref{sec:sephyp}, we discuss a
recursion for regions. The recursion is found by adding all possible
first columns to Shi tableaux for regions in $\shi{}{n-1}{m}$ to
create all Shi tableaux for regions in $\shinm$. The generating
function keeps track of the possible first columns and rows.  We use
two statistics $\rStat{}$ and $\cStat{}$ on regions in the extended
Shi arrangement. Let $\region{}\in\shinm$ and define

$$\rStat{\region{}}=|\{(j,k):\region{}\textrm{ and
}\Afund\textrm{ are separated by }\Halpha{1}{j}{k}\textrm{ and
}1\leq k\leq m\}|$$ and
$$\cStat{\region{}}=|\{(i,k):
\region{}\textrm{ and }\Afund\textrm{ are separated by
}\Halpha{i}{n-1}{k}\textrm{ and }1\leq k\leq m\}|.$$

$\rStat{\region{}}$ counts the number of translates of
$\Hak{\a_{1j}}{0}$ which separate $\region{}$ from $\Afund$, for
$1\leq j\leq n-1$. Similarly for $\cStat{\region{}}$ and translates of $\Hak{\a_{i,n-1}}{0}$.

The generating function is
$$\wP{n}{m}{\aij}{p}{q}=\sum_{\region{}\in\sepSet{n}{m}{\aij}}p^{\cStat{\region{}}}q^{\rStat{\region{}}}.$$
\begin{example}
$\wP{3}{2}{\a_1}{p}{q}=p^4q^2+p^4q^3+p^4q^4.$
\end{example}

We let $\polyInt{k}{p,q}=\sum_{j=0}^{k-1}p^jq^{k-1-j}$ and
$\polyInt{k}{q}=\polyInt{k}{1,q}$. We will also need to truncate
polynomials and the notation we use for that is
$$\trunc{\sum_{j=0}^{j=n}a_jq^j}{q}{N}=\sum_{j=0}^{j=N}a_jq^j.$$

The statistics are related to the $n$-core partition assigned by $\RiMap$ to the
$m$-minimal alcove for the region.

\begin{proposition}
\label{prop:cAndr}
Let $\lambda$ be an $n$-core with vector of level numbers
$(b_0,\ldots,b_{n-1})$ and suppose $\RiMap(\lambda)=\region{}$ and
$\region{}\in\sepSet{n}{m}{\theta}$. Then
$\rStat{\region{}}=m+\sum_{i=2}^{n-1}b_i=\ell(\lambda)$ and
$\cStat{\region{}}=m+\sum_{i=2}^{n-1}(m-1-b_i)=\lambda_1$.
\end{proposition}
\begin{proof}

Let $\lambda$, $(b_0,\ldots,b_{n-1})$, and $\region{}$ be as in the
statement of the claim. Let $\{e_{ij}\}$ be the region coordinates for
$\region{}$ and $\{k_{ij}\}$ be the coordinates of $\region{}$'s
$m$-minimal alcove, and let $\{p_i\}$ and $\{\tilde{p}_i\}$ be as in
the definition of $\RiMap$.  Then
\begin{eqnarray*}
\rStat{\region{}}&=&e_{1,n-1}+e_{1,n-2}+\ldots+e_{11}\\
&=&k_{1,n-1}+k_{1,n-2}+\ldots+k_{11}\\ &=&\lfloor
\frac{p_n}{n}\rfloor+\ldots+\lfloor \frac{p_1}{n}\rfloor\\ &=&\lfloor
\frac{\tilde{p}_n}{n}\rfloor+\ldots+\lfloor
\frac{\tilde{p}_1}{n}\rfloor\\ &=&\sum_{i=0}^{n-1}b_i
\end{eqnarray*}
The second part of the claim follows since
$\cStat{\region{}}=k_{1,n-1}+k_{2,n-1}+\ldots+k_{n-1,n-1}$ and
$k_{i,n-1}=(m-1)-k_{1,i-1}$ for $\region{}\in\sepSet{n}{m}{\theta}$
and $2\leq i\leq n-1$.
\end{proof}

We can also relate the statistics $\rStat{}$ and $\cStat{}$ to the
$n$-core partition corresponding under $\FVMap$ to the $m$-minimal
alcove of the region $\region{}$.

For now, let $\kwa$ be the Shi coordinate of $\wA$.
Note $\kwa < \brac{w^{-1}(\frac 1n \rho)}{\a} < \kwa +1$,
so $\kwa = \floor{\brac{w^{-1}(\frac 1n \rho)}{\a}}$.
\begin{proposition}
\label{prop:cAndr2}
Let $\lambda$ be an $n$-core and $w \in \coset$
be of minimal length such that $\lambda = w \emptyset$.
\begin{enumerate}
\item
Then $\sum_i  \kwaij{i}{n-1} = \lambda_1$.
\item
Then $\sum_j  \kwaij{1}{j} = \ell(\lambda)$.
\end{enumerate}
\end{proposition}

\begin{proof}
Consider
\begin{eqnarray*}
\brac{w^{-1}(\frac 1n \rho)}{\alij{i}{n}} &=&
\brac{u^{-1}(-\beta)}{\alij{i}{n}}  + \brac{u^{-1} (\frac 1n \rho)}{\alij{i}{n}}
\\
&=&\brac{u^{-1}(-\beta)}{\alij{i}{n}}  + \brac{\frac 1n \rho}{\e_{u(i)} - \e_{u(n)}}.
\\
&=&\brac{u^{-1}(-\beta)}{\alij{i}{n}}  + \brac{\frac 1n \rho}{\e_{u(i)} - \e_r}.
\end{eqnarray*}
Note $\brac{u^{-1}(-\beta)}{\alij{i}{n}}  \in \ZZ$ and
$\floor{\brac{\frac 1n \rho}{\e_{u(i)} - \e_r}} = 0 $ if $u(i) < r$, but
$\floor{\brac{\frac 1n \rho}{\e_{u(i)} - \e_r}} = -1 $ if $u(i) > r$.
Hence
$$\sum_i \floor{\brac{u^{-1} (\frac 1n \rho)}{\alij{i}{n}}} = -(n-r).$$
We then compute
\begin{eqnarray*}
\sum_i \kwaij{i}{n}
&=& r-n + \sum_i \brac{u^{-1}(-\beta)}{\alij{i}{n}}
\\
&=& r-n + \brac{u^{-1}(-\beta)}{\gamma} = r-n +qn
\\
&=& \lambda_1
\end{eqnarray*}
by the computations in the proof of Proposition \ref{prop-gamma}.

Likewise, $\floor{\brac{u^{-1} (\frac 1n \rho)}{\alij{1}{j}} }
 =\floor{\brac{\frac 1n \rho}{ \e_{s} - \e_{u(j)} }} = -1$
if $u(j) < s$ and zero otherwise.
As above,
\begin{eqnarray*}
\sum_j \kwaij{1}{j} &=& -(s-1) + \sum_j \brac{u^{-1}(-\beta)}{\alij{1}{j}}
\\
&=& 1-s + \brac{u^{-1}(-\beta)}{\Gamma} = 1-s +Mn
\\
&=& \ell(\lambda)
\end{eqnarray*}
by the computations in the proof of Proposition \ref{prop-Gamma}.
\end{proof}

We thus obtain another corollary to Theorem~\ref{thm:baseCase}.
\begin{corollary}
\label{cor:boundaryRegionsGF}
$$\wPthm= p^mq^m\polyInt{m}{p,q}^{n-2}.$$
\end{corollary}
\begin{proof}
Corollary~\ref{cor:boundaryRegionsGF} follows from
Theorem~\ref{thm:baseCase} or \ref{thm:altBaseCase},
Proposition~\ref{prop:cAndr}, and the abacus representation of
$n$-cores which have the prescribed hook length.
\begin{eqnarray*}
\wPthm&=&\sum_{\region{}\in\sepSet{n}{m}{\theta}}p^{\cStat{\region{}}}q^{\rStat{\region{}}}\\
&=&\sum_{\substack{\lambda\textrm{ is an }n-\textrm{core}\\h^{\lambda}_{11}=n(m-1)+1}}p^{m+\sum_{i=2}^{n-1}b_i}q^{m+\sum_{i=2}^{n-1}(m-1-b_i)}\\
&=&\sum_{\substack{(b_2,\ldots,b_{n-1})\\0\leq b_i\leq
m-1}}p^mq^m\left(\prod_{i=2}^{n-1}p^{b_i}q^{m-1-b_i}\right)\\
&=&p^mq^m(p^{m-1}+p^{m-2}q+\cdots+pq^{m-2}+q^{m-1})^{n-2}\\
&=&p^mq^m\polyInt{m}{p,q}^{n-2}.
\end{eqnarray*}

\end{proof}

In particular, by evaluating at $p=q=1$,
we have the following corollary to  Corollary~\ref{cor:boundaryRegionsGF}.
\begin{corollary}
There are $m^{n-2}$ regions in $\shinm$ which have $\Hthm=\Halpha{1}{n-1}{m}$
as a separating wall.
\label{cor:boundaryRegions}
\end{corollary}

There are direct explanations for Corollary~\ref{cor:boundaryRegions},
but we need Theorem~\ref{thm:baseCase},
Theorem~\ref{thm:altBaseCase}, and Corollary~\ref{cor:boundaryRegionsGF} to develop our recursions, where we
need to know more than the number of regions which have $\Hthm$ as a
separating wall. We use the number of hyperplanes which separate each
region from the origin.

\section{Arbitrary separating wall}
\label{sec:sephyp}

The next few lemmas provide an inductive method for determining
whether or not $\region{}\in\shi{}{n}{m}$ is an element of
$\sepSet{n}{m}{\a_{2,n-1}}$.

Given a Shi tableau $T_{\region{}} = \{e_{ij}\}_{1\leq i\leq j\leq
  n-1}$, where $\region{}\in\shi{}{n}{m}$, let $\tilde{T}_{\region{}}$
  be the tableau with entries $\{e_{ij}\}_{1\leq i\leq j\leq
  n-2}$. That is, $\tilde{T}_{\region{}}$ is $T_{\region{}}$ with the
  first column removed.
\omitt{
\begin{example}
Suppose $\region{}\in\shi{}{5}{m}$ and\\
\begin{tikzpicture}
\node at (-.6,0) {$T_{\region{}}=$};
\youngDiagram{{{$e_{14}$, $e_{13}$, $e_{12}$,
    $e_{11}$}, {$e_{24}$, $e_{23}$,
    $e_{22}$},{$e_{34}$, $e_{33}$}, {$e_{44}$}}}{.6}
\node at (3,0){.  Then };
\begin{scope}[shift={(4.7,0)}]
\node at (-.6,0) {$\tilde{T}_{\region{}}=$};
\youngDiagram{{{$e_{13}$, $e_{12}$,
    $e_{11}$}, { $e_{23}$,
    $e_{22}$},{$e_{33}$}}}{.6}
\end{scope}.
\end{tikzpicture}
\end{example}
} 
The next lemma tells us that $\tilde{T}_{\region{}}$ is always the Shi
tableau for a region in one less dimension.
\begin{lemma}
\label{lemma:tabRest}
If $T_{\region{}}$ is the tableau of a region $\region{}\in\shi{}{n}{m}$ and
$1\leq u\leq v\leq\n-1$, then $\tilde{T}_{\region{}}=T_{\tilde{\region{}}}$ for some
$\tilde{\region{}}\in\shi{}{n-1}{m}$.
\end{lemma}
\begin{proof}
This follows from Lemma~\ref{lemma:shiTabReg}.
\end{proof}

\begin{lemma}
\label{prop:induction}
  Let $T_{\region{}}$ be the Shi tableau for the region
  $\region{}\in\shi{}{n}{m}$ and let $\tR$ be defined by
$T_{\tR}=\tilde{T}_{\region{}}$, where
  $\tR\in\shi{}{n-1}{m}$ by Lemma~\ref{lemma:tabRest}. Then
  $\region{}\in\sepSet{n}{m}{\a_{i,n-2}}$ if and only if
  $\tR\in\sepSet{n-1}{m}{\a_{i,n-2}}$.
\end{lemma}
\begin{proof}
This follows from Lemma~\ref{lem:indecomp}.
\end{proof}

In terms of generating functions, Lemma~\ref{prop:induction}
states:

\begin{eqnarray}
\label{eqn:gfRec1}
\wP{n}{m}{\a_{i,n-2}}{p}{q}&=&\sum_{\region{}\in\sepSet{n}{m}{\a_{i,n-2}}}p^{\cStat{\region{}}}q^{\rStat{\region{}}}\\
&=&\sum_{\R1\in\sepSet{n-1}{m}{\a_{i,n-2}}}\sum_{\substack{\region{}\in\shi{}{n}{m}\\\tR=\R1}}p^{\cStat{\region{}}}q^{\rStat{\region{}}}\nonumber
\end{eqnarray}

If $\R1\in\sepSet{n-1}{n}{\a_{i,n-}}$ and $\region{}\in\shinm$ are
such that $\tR=\R1$, then, since $e_{i,n-2}=m$ in the Shi tableau for
$\R1$, $\rStat{\region{}}=\rStat{\R1}+m$ and
$\cStat{\region{}}=\cStat{\R1}+k$, for some $k$. We need to establish
the possible values for $k$.

We will use Proposition 3.5 from \cite{Richards} to do this.  His
``pyramids'' correspond to our Shi tableaux for regions, with his $e$ and $w$
being our $n$ and $m+1$. He does not mention hyperplanes, but with the
conversion $_ua_v=m-e_{u+1,v}$ his conditions in Proposition 3.4 become
our conditions in Lemma~\ref{lemma:shiTabReg}.

In our language, his Proposition 3.5 becomes

\begin{lemma}[\cite{Richards}]
\label{prop:stairShi}
Let $\mu_1, \mu_2,\ldots,\mu_{n}$ be non-negative integers with
$$\mu_1\geq \mu_2\geq \ldots \geq \mu_{n}=0\text{ and }\mu_{i}\leq
(n-i)m.$$ Then there is a unique region $\region{}\in\shi{}{n}{m}$
with Shi tableau $T_{\region{}}=\{e_{ij}\}_{1\leq i\leq j\leq \n-1}$ such
that
$$\mu_j=\mu_j(\region{})=\sum_{i=1}^{n-j}e_{i,n-j}\text{ for }1\leq j\leq n-1$$
\end{lemma}
We include his proof for completeness.
\begin{proof}
  By Lemma~\ref{lemma:shiTabReg}, we have $e_{ij}\geq e_{i+1,j}$ and
  $e_{ij}\geq e_{i,j-1}$ for $1\leq i<j\leq n-1$, which, combined with
  $0\leq e_{ij}\leq m$, means that the column sums
  $\mu_j=\sum_{i=1}^{j}e_{ij}$ form a partition such that $0\leq
  \mu_j\leq m(n-j)$.

\medskip

We use induction on $n$ to show that given such a partition $\mu$,
there is at most one region whose Shi tableau has column sums $\mu$. It is
clearly true for $n=2$. Let $n>2$ and suppose we had two regions
$\region{1}$ with coordinates $\{e_{ij}\}_{1\leq i\leq j\leq n-1}$
and $\region{2}$ with coordinates $\{f_{ij}\}_{1\leq i\leq j\leq n-1}$
such that
$$\mu_j=\sum_{i=1}^je_{ij}=\sum_{i=1}^jf_{ij}.$$ By induction
$e_{ij}=f_{ij}$ for $1\leq i\leq j\leq n-2.$ Let $u$ be the least
index such that $e_{u,n-1}\neq f_{u,n-1}$ and assume
$e_{u,n-1}<f_{u,n-1}$. Then since
$\sum_{i=1}^{n-1}e_{i,n-1}=\sum_{i=1}^{n-1}f_{i,n-1}$, we have that
$e_{v,n-1}>f_{v,n-1}$ for some $v$ such that $u<v\leq n-1$. Then since
$f_{u,n-1}\leq f_{u,v-1}+f_{v,n-1}+1$ by Lemma~\ref{lemma:shiTabReg} and
$f_{u,v-1}=e_{u,v-1}$ by induction, we have
$$e_{u,n-1}< f_{u,n-1}\leq e_{u,v-1}+f_{v,n-1}+1\leq
e_{u,v-1}+e_{v,n-1}.$$
This contradicts Lemma~\ref{lemma:shiTabReg} applied to $\region{1}$.

\medskip
However, there are $\frac{1}{mn+1}{{(m+1)n}\choose n}$ dominant Shi
regions by \cite{Shi1997} for $m=1$ and \cite{A2004} for $m > 1$ and
it is well-known that there are also $\frac{1}{mn+1}{{(m+1)n}\choose
n}$ partitions $\mu$ such that $0\leq \mu_i\leq m(n-i)$, so we are
done.
\end{proof}

\begin{example}
Consider $\region{1}$, $\region{2}$, and $\region{3}$ in $\shi{}{3}{2}$
with tableaux
\begin{center}
\begin{tikzpicture}[font=\footnotesize]
\begin{scope}[shift={(-5,0)}]
\youngDiagram{{{2,2,1},{2,2},{2}}}{.4}
\end{scope}
\begin{scope}[shift={(-3,0)}]
\youngDiagram{{{2,2,1},{2,2},{1}}}{.4}
\end{scope}
\begin{scope}[shift={(-1,0)}]
\youngDiagram{{{2,2,1},{2,2},{0}}}{.4}
\end{scope}
\end{tikzpicture}
\end{center}
respectively. Then
$\tilde{\region{1}}=\tilde{\region{2}}=\tilde{\region{2}}=\region{}$,
where $\region{}$ is the region in $\shi{}{2}{2}$ with tableau
\begin{tikzpicture}[font=\footnotesize]
\youngDiagram{{{2,1},{2}}}{.5}
\end{tikzpicture}
\end{example}

Let $\a=\aij$, where $1\leq i\leq j\leq n-2$ in the
following.  Suppose $\R1$ is a region,
where $\R1\in\shi{}{n-1}{m}$ and $T_{\R1}=\{e_{ij}\}_{1\leq i\leq j\leq
n-2}$, and $k$ is an integer such that $\sum_{i=1}^{n-2}e_{i,n-2}\leq
k\leq (n-1)m$. Then 
Lemma~\ref{prop:stairShi} means
there is a region $\region{}\in\shi{}{n}{m}$ such that
$\tR=\R1$ and the first column sum of $\region{}$'s Shi tableau is $k$.
Additionally, by Lemma~\ref{prop:induction}, we have
$\R1\in\sepSet{n-1}{m}{\a}$ if and only if
$\region{}\in\sepSet{n}{m}{\a}$. On the other hand, given
$\region{}\in\shi{}{n}{m}$ with Shi tableau
$T_{\region{}}=\{e_{ij}\}_{1\leq i\leq j\leq n-1}$, let $k$ be the
first column sum of $T_{\region{}}$. Then by Lemma~\ref{lemma:tabRest}
and the fact that $e_{i,n-1}\geq e_{i,n-2}$ for $1\leq i\leq n-2$, the
pair $(\tR,k)$ is such that $\tR\in\shi{}{n-1}{m}$ and the first
column sum of $T_{\tR}$ is not more than $k$. Again, by
Lemma~\ref{prop:induction}, we have $\tR\in\sepSet{n-1}{m}{\a}$ if and
only if $\region{}\in\sepSet{n}{m}{\a}$.

We continue (\ref{eqn:gfRec1}), keeping in mind that
$\cStat{\R1}$ is the first column sum for $T_{\R1}$. For ease of reading, write $\a$ for $\a_{i,n-2}$ in the following calculation.
\begin{eqnarray*}
\wP{n}{m}{\a}{p}{q}&=&\sum_{\region{}\in\sepSet{n}{m}{\a}}p^{\cStat{\region{}}}q^{\rStat{\region{}}}\\\nonumber
&=&\sum_{\R1\in\sepSet{n-1}{m}{\a}}\sum_{\substack{\region{}\in\shi{}{n}{m}\\\tR=\R1}}p^{\cStat{\region{}}}q^{\rStat{\region{}}}\\\nonumber
&=&\sum_{\R1\in\sepSet{n-1}{m}{\a}}\sum_{\substack{k\\\cStat{\R1}\leq
k\leq n(m-1)}}p^{k}q^{\rStat{\R1}+m}\\\nonumber
&=&\sum_{\R1\in\sepSet{n-1}{m}{\a}}\sum_{\substack{k'\\0\leq k'\leq
n(m-1)-\cStat{\R1}}}p^{\cStat{\R1}+k'}q^{\rStat{\R1}+m}\\\nonumber
&=&\trunc{\sum_{\R1\in\sepSet{n-1}{m}{\a}}\sum_{\substack{k'\\0\leq
k'\leq
n(m-1)}}p^{\cStat{\R1}+k'}q^{\rStat{\R1}+m}}{p}{(n-1)m}\\\nonumber
&=&\trunc{\sum_{\R1\in\sepSet{n-1}{m}{\a}}\sum_{\substack{k'\\0\leq
k'\leq
n(m-2)}}p^{\cStat{\R1}+k'}q^{\rStat{\R1}+m}}{p}{(n-1)m}\\\nonumber
&\quad&\text{since }\cStat{\R1}\geq m\\\nonumber
&=&\trunc{q^m\left(\sum_{\R1\in\sepSet{n-1}{m}{\a}}
p^{\cStat{\R1}}q^{\rStat{\R1}} \right)\left( \sum_{\substack{k'\\0\leq
k'\leq n(m-2)}}p^{k'}\right)}{p}{(n-1)m}\\\nonumber
&=&\trunc{q^m\polyInt{(n-2)m+1}{p}\wP{n-1}{m}{\a}{p}{q}}{p}{(n-1)m}.
\end{eqnarray*}

The result of the above calculation is that
\begin{equation}
\label{eqn:gfRec2}
\wP{n}{m}{\a}{p}{q}=\trunc{q^m\polyInt{(n-2)m+1}{p}\wP{n-1}{m}{\a}{p}{q}}{p}{(n-1)m}
\end{equation}
when $\a=\a_{i,n-2}$.

The next proposition will provide a method for determining whether or
not $\HH{1}{n-j}{m}$ is a separating wall for $\region{}$. Given a Shi
tableau $T=\{e_{ij}\}_{1\leq i\leq j\leq\n-1}$ for a region in
$\shi{}{n}{m}$, let $T'$ be its {\it conjugate} given by
$T'=\{e'_{ij}\}_{1\leq i\leq j\leq\n-1}$, where $e'_{ij}=e_{n-j,n-i}.$
\omitt{
\begin{example}
\end{example}
}
By Lemma~\ref{lemma:shiTabReg}, $T'$ will also be Shi tableau of a
region in $\shi{}{n}{m}.$ Additionally, by Lemma~\ref{lem:indecomp},
we have the following proposition.

\begin{proposition}
\label{prop:symmetry}
Suppose the regions $\region{}$ and $\region{}'$ are related by
$$(T_{\region{}})'=T_{\region{}'}.$$ Then
$\region{}\in\sepSet{n}{m}{\a_{ij}}$ if and only if
$\region{}'\in\sepSet{n}{m}{\a_{n-j,n-i}}$.
\end{proposition}

In terms of generating functions, this becomes the following:
\begin{equation}
\label{eqn:symmetry}
\wP{n}{m}{\aij}{p}{q}=\wP{n}{m}{\a_{n-j,n-i}}{q}{p}.
\end{equation}

We will now combine Theorem~\ref{thm:baseCase},
Proposition~\ref{prop:induction}, and Proposition~\ref{prop:symmetry}
to produce an expression for the generating function for regions with
a given separating wall.

Given a polynomial $f(p,q)$ in two variables, let $\colRem{k}{m}(f(p,q))$
be the polynomial
$$\trunc{q^m\polyInt{m(k-2)+1}{p}f(p,q)}{p}{(k-1)m}.$$
We define the polynomial $\rho(f)$ by $$\rho(f)(p,q)=f(q,p).$$
 \omitt{and let
$\rho(f)$ be the original polynomial with $p$ and $q$ reversed: $f(q,p)$.}
Then (\ref{eqn:gfRec2}) is

$$\wP{n}{m}{\aij}{p}{q}=\colRem{n}{m}(\wP{n-1}{m}{\aij}{p}{q})$$ for $j=n-2$ and
(\ref{eqn:symmetry})
is $$\wP{n}{m}{\aij}{p}{q}=\rho(\wP{n}{m}{\a_{n-j,n-i}}{p}{q}).$$

Finally, the full recursion is
\begin{theorem}
$$\wP{n}{m}{\a_{uv}}{p}{q}=\colRem{n}{m}(\colRem{n-1}{m}(\ldots\colRem{v+2}{m}(\rho(\colRem{v+1}{m}(\ldots(\colRem{v-u+3}{m}(p^mq^m\polyInt{m}{p,q}^{v-u})\ldots).$$
\end{theorem}

The idea behind the theorem is that, given a root $\a_{uv}$ in dimension
$n-1$, we remove columns using Lemma~\ref{prop:stairShi} until we are in
dimension $(v+1)-1$, then we conjugate, then remove columns again until
our root is $\a_{1,v-u+1}$ and we are in dimension $(v-u+2)-1$.

\begin{example}
We would like to know how many elements there are in
$\sepSet{7}{2}{\a_{24}}$; that is, how many dominant regions in the
$2$-Shi arrangement for $n=7$ have $\Hak{\a_{24}}{2}$ as a separating
wall. In order to make this readable, we omit the $m$ subscript, since
it is always $2$ in this calculation.

\begin{eqnarray*}
\wP{7}{2}{\a_{24}}{p}{q}&=&\trunc{q^2\polyInt{11}{p}\wP{6}{}{\a_{24}}{p}{q}}{p}{12}\\
&=&\trunc{q^2\polyInt{11}{p}\trunc{q^2\polyInt{9}{p}\wP{5}{}{\a_{24}}{p}{q}}{p}{10}}{p}{12}\\
&=&\trunc{q^2\polyInt{11}{p}\trunc{q^2\polyInt{9}{p}\wP{5}{}{\a_{13}}{q}{p}}{p}{10}}{p}{12}\\
&=&\trunc{q^2\polyInt{11}{p}\trunc{p^2\polyInt{9}{p}\trunc{q^2\polyInt{7}{q}\wP{4}{}{\a_{13}}{q}{p}}{q}{8}}{q}{10}}{q}{12}\\
&=&\trunc{q^2\polyInt{11}{p}\trunc{p^2\polyInt{9}{p}\trunc{q^2\polyInt{7}{q}\left(p^2q^2\polyInt{2}{p,q}^2\right)}{q}{8}}{p}{10}}{p}{12}\\
\end{eqnarray*}
After expanding this polynomial and evaluating at $p=q=1$, we see
there are 781 regions in the dimension 7 2-Shi arrangement which have
$\Hak{\a_{24}}{2}$ as a separating wall.
\end{example}

\section*{Future work}
It would be interesting to expand this problem by considering a given
set of more than one separating walls.  That is, given a set
${\mathcal H}_{\Delta'} =\{ H_{\alpha,m} , \alpha \in \Delta'
\subseteq \Delta\}$ of hyperplanes in the Shi arrangement, find the
number of regions having all the hyperplanes in $\mathcal H_{\Delta'}$
as separating walls.

We would again be able to define a similar generating function, use
the functions $\phi_{k,m}$ and $\rho$ corresponding to truncation and
conjugation of the Shi tableaux, but we should be able to
compute the generating function for a suitably chosen base case.

\section*{Acknowledgement}
We thank Matthew Fayers for telling us of \cite{Richards} and
explaining its relationship to \cite{FV1}. We thank Alessandro
Conflitti for simplifying the proof of
Proposition~\ref{prop:FVBij}. We thank the referee for comments which
helped us improve the exposition.  \bibliographystyle{abbrvnat}
\bibliography{FTVBib}
\label{sec:biblio}

\end{document}